\documentclass[11pt]{amsart}
\usepackage{geometry}                
\usepackage{graphicx}
\usepackage{amssymb}
\usepackage{epstopdf}
\usepackage{biblatex}
\begin{document}

\title{Pinched self-dual Weyl curvature on Einstein four-manifolds}


\author{Inyoung Kim}
\maketitle
\begin{abstract}
We show that a compact oriented riemannian four-manifold with harmonic and pinched self-dual Weyl curvature
is anti-self-dual if the type is nonpositive. 
 The main part is to show that there is an almost-K\"ahler structure
outside the zero set of the self-dual Weyl curvature. 
\end{abstract}

\maketitle

\section{\large\textbf{Introduction}}\label{S:Intro}
Let $(M, g)$ be an oriented riemannian four-manifold and let
\[R(X, Y)Z:=\nabla_{X}\nabla_{Y}Z-\nabla_{Y}\nabla_{X}Z-\nabla_{[X, Y]}Z,\]
where $\nabla$ is the Levi-Civita connection. 
 Then there is a symmetric bilinear map 
 $R:\Lambda^{2}M\times\Lambda^{2}M\to \mathbb{R}$ such that 
 \[R(X\wedge Y, V\wedge W):=g(R(X, Y)W, V),\]
 where 
 $\Lambda^{2}M$ is the set of bivectors. 
 Then the curvature operator is defined by 
 \[g(\mathfrak{R}(X\wedge Y), V\wedge W)=R(X\wedge Y, V\wedge W).\]
According to the decomposition of $\Lambda^{2}=\Lambda^{+}\oplus \Lambda^{-}$,
the curvature operator $\mathfrak{R}:\Lambda^{2}\to\Lambda^{2}$ is given by 
\[ 
\mathfrak{R}=
\LARGE
\begin{pmatrix}

 \begin{array}{c|c}
\scriptscriptstyle{W^{+}\hspace{5pt}\scriptstyle{+}\hspace{5pt}\frac{s}{12}}I& \scriptscriptstyle{ric_{0}}\\
 \hline
 
 \hspace{10pt}
 
 \scriptscriptstyle{ric_{0}}& \scriptscriptstyle{W^{-}\hspace{5pt}\scriptstyle{+}\hspace{5pt}\frac{s}{12}}I\\
 \end{array}
 \end{pmatrix}
 \]
 where $s$ is the scalar curvature. 
 The Hodge star operator is defined by $\left<*\omega_{1}, \omega_{2}\right>d\mu=\omega_{1}\wedge\omega_{2}$ 
 for $k$-forms $\omega_{1}, \omega_{2}$ on $M$, where 
 $d\mu$ is the volume form of $(M, g)$.
 A 2-form $\omega\in\Lambda^{+}$ is self-dual if $*\omega=\omega$
 and $\omega\in\Lambda^{-}$ is anti-self-dual if $*\omega=-\omega$.
When $ric_{0}=0$, $(M, g)$ is Einstein. 
 When $W^{\pm}=0$, $g$ is an anti-self-dual(self-dual) metric respectively.

 \vspace{30pt}

\hrule

\vspace{10pt}

$\mathbf{Keywords}$: four-manifold, almost-K\"ahler, self-dual Weyl curvature, Einstein metric

$\mathbf{MSC}$: 53C20, 53C21, 53C55

\vspace{10pt}

Republic of Korea, 

Email address: kiysd5@gmail.com

Let $(M, g)$ be an almost-Hermitian four-manifold. 
Then for tangent vectors $X, Y\in T_{p}M$, $g(X, Y)=g(JX, JY)$. 
The fundamental 2-form is defined by 
\[\omega(X, Y):=g(JX, Y).\]
When $d\omega=0$, $(M, g, \omega, J)$ is almost-K\"ahler. 
If $J$ is integrable, then $(M, g, \omega, J)$ is Hermitian. 
When $J$ is integrable and $d\omega=0$, $(M, g, \omega, J)$ is K\"ahler. 
We note that on K\"ahler manifolds, $\nabla\omega=0$, where $\nabla$ is the Levi-Civita connection of $(M, g)$.

Let $(M, g, \omega, J)$ be an almost-Hermitian four-manifold and let $\{X, JX, Y, JY\}$ be an orthonormal basis. 
The holomorphic sectional curvature is defined by 
\[H(X):=R(X\wedge JX, X\wedge JX)\] 
and the orthogonal holomorphic sectional curvature is defined by 
\[H_{ohol}(X, Y):=\frac{1}{2}(R(X\wedge JX, X\wedge JX)+R(Y\wedge JY, Y\wedge JY)).\]
Let $H_{omax}(H_{omin})$ is the maximum(minimum) of the orthogonal holomorphic sectional curvature 
and $\lambda_{i}$ be the eigenvalues of $W^{-}$ such that
$\lambda_{1}\leq\lambda_{2}\leq\lambda_{3}$. 
We show that 
\[H_{omax}=\frac{s+3s^{*}}{24}+\frac{\lambda_{3}}{2}, \hspace{5pt}  H_{omin}=\frac{s+3s^{*}}{24}+\frac{\lambda_{1}}{2}, \]
where $s^{*}=2R(\omega, \omega)$.
For a K\"ahler-Einstein surface, we get
\[H_{max}=\frac{s}{6}+\frac{\lambda_{3}}{2}, \hspace{5pt}  H_{min}=\frac{s}{6}+\frac{\lambda_{1}}{2}, \]
where $H_{max}(H_{min})$ is the maximum(minimum) of the holomorphic sectional curvature. 

Let $M$ be a compact K\"ahler-Einstein surface with nonpositive bisectional curvature.
Suppose
\[H_{av}-H_{min}\leq a(H_{max}-H_{min}),\]
for some $a<\frac{2}{3(1+\sqrt{6/11})}$.
Here 
\[H_{av}(p):=\frac{1}{Vol(\mathbb{S}^{3})}\int_{S_{p}}Hda,\]
where $S_{p}$ is the unit sphere of $T_{p}M$. 
By Berger's Theorem [5], $H_{av}(p)=\frac{s}{6}$ for a K\"ahler surface.
Siu-Yang showed that $M$ is biholomorphically isometric to a compact quotient of the complex 2-ball with an invariant metric [26]. 
Improving results by Chen, Hong and Yang [7], Guan [10] showed that a compact K\"ahler-Einstein surface with nonpositive scalar curvature is a compact quotient
of the complex 2-dimensional unit ball or plane if 
\[H_{av}-H_{min}\leq \frac{1}{2}(H_{max}-H_{min}).\]

It was shown that a certain bounded differentiable function on $M-N$ is superharmonic,
where $N$ is the set of ball-like points and 
it was mentioned that such a function 
can be extended as a superharmonic function on $M$ [26]. 
The coefficient was improved by $1/2$
by using a different superharmonic function, which was called Hong Cang Yang's function [10]. 
This function may not be differentiable on $M-N$. 
Moreover, the result was shown without a hypothesis on the nonpositive bisectional curvature [10]. 
The function is $\Phi=3B-A$ on K\"ahler-Einstein surfaces, where 
$A$ and $B$ are given by  
\[H_{max}=H_{min}+\frac{1}{2}(A+B)\]
\[H_{av}=H_{min}+\frac{1}{3}A.\]
In terms of $\lambda_{i}$, we have 
\[A=-\frac{3}{2}\lambda_{1}, \hspace{5pt} B=\lambda_{3}+\frac{1}{2}\lambda_{1},\]
where $\lambda_{i}$ are the eigenvalues of $W^{-}$ such that $\lambda_{1}\leq\lambda_{2}\leq\lambda_{3}$.
Then $\Phi=3B-A=-3\lambda_{2}$,
which was mentioned in [10]. 
The set $A=B$ in [10] is the set where $2\lambda_{1}+\lambda_{3}=0$.

Let $\delta W^{+}:=-\nabla\cdot W^{+}$. 
A four-dimensional oriented riemannian manifold $(M, g)$ is said to have harmonic self-dual Weyl curvature if $\delta W^{+}=0$.
Let $(M, g)$ be a compact riemannian four-manifold with harmonic self-dual Weyl curvature and nonpositive scalar curvature. 
Let $\lambda_{i}$ be the eigenvalues of $W^{+}$ such that $\lambda_{1}\leq\lambda_{2}\leq\lambda_{3}$. Suppose 
\[-8\left(1-\frac{\sqrt{3}}{2}\right)\lambda_{1}\leq \lambda_{3}\leq -2\lambda_{1}.\]
Polombo showed that $(M, g)$ is anti-self-dual [23]. We note that 
$-8\left(1-\frac{\sqrt{3}}{2}\right)\lambda_{1}\leq \lambda_{3}$ implies $\lambda_{1}+\lambda_{3}\geq 0$. 
It was shown that the zero set $W^{+}$ has codimension $\geq 2$ for a compact, oriented four-manifold
with harmonic self-dual Weyl curvature [3], [23]. 

We note that Guan's pinching condition 
is equivalent to
\[\lambda_{1}+\lambda_{3}\geq 0,\]
where $\lambda_{i}$ are the eigenvalues of $W^{-}$ such that $\lambda_{1}\leq\lambda_{2}\leq\lambda_{3}$. 
Since $\lambda_{1}+\lambda_{2}+\lambda_{3}=0$, this is equivalent to $\lambda_{2}\leq 0$. 
Since $\lambda_{1}\leq 0$ and $\lambda_{3}\geq 0$, this condition is equivalent to $det W^{-}=\lambda_{1}\lambda_{2}\lambda_{3}\geq 0$.  
Compact riemannian four-manifolds with harmonic self-dual Weyl curvature with $det W^{+}>0$ were classified by LeBrun [20] and Wu [27]. 
Moreover, LeBrun showed that the condition $det W^{+}>0$ can be improved by $\lambda_{2}\leq\frac{\lambda_{3}}{4}$ and $W^{+}$ is nowhere zero. 
It was shown by LeBrun that the largest eigenvalue of $W^{+}$ is a smooth function if $det W^{+}>0$ [20]. 
In particular, this implies $\lambda_{3}$ is a smooth function on $M-N$ if $det W^{+}\geq 0$, where $N$ is the zero set of $W^{+}$.
By changing the orientation, we get the same result with respect to $W^{-}$.
Then this implies that in case of K\"ahler-Einstein surfaces, $H_{max}$ is a smooth function on $M-N$, 
where $N$ is the set of ball-like points.

In [15], it was shown that a compact oriented riemannian four-manifold with harmonic self-dual Weyl curvature is anti-self-dual
if $det W^{+}\geq 0$ and the scalar curvature is nonpositive.
Instead of considering the set $A=B$, the set $M_{W}$, where the distinct number of eigenvalues of $W^{+}$ is constant [8], was considered. 
$M_{W}$ is a dense open set and $\lambda_{i}$ are differentiable on $M_{W}$ [8]. 
It was shown that that the function 
$\Phi^{1/3}$, where $\Phi=-\lambda_{2}$ is strictly superharmonic on $M_{W}$ and $\Phi^{-2}|d\Phi|^{2}=0$ on $M_{W}$, which is a contradiction unless $M=N$.
This generalizes Guan's Theorem and Polombo's Theorem.

In this article, we consider compact oriented riemannian four-manifolds with harmonic self-dual Weyl curvature
and $\lambda_{2}\leq\frac{\lambda_{3}}{4}$, where $\lambda_{i}$ are the eigenvalues of $W^{+}$ such that $\lambda_{1}\leq\lambda_{2}\leq\lambda_{3}$. 
Define a metric $g:=(\lambda_{3})^{2/3}h$ on $M-N$, where $N$ is the zero set of $W^{+}$. 
Following [20], instead of a function $\Phi$, we use an eigenform with the largest eigenvalue of $W^{+}$ with respect to $g$.
Working with the eigenform $\omega$ with the eigenvalue $\lambda_{3}$, we only need to consider the zero set. 
When we assume $det W^{+}\geq 0$, we get a K\"ahler structure. 
On the other hand, if we assume a weaker condition $\lambda_{2}\leq\frac{\lambda_{3}}{4}$, we get an almost-K\"ahler structure.
Then by the argument of LeBrun using almost-K\"ahler geometry [18], [19], 
we get that $W^{+}$ has a degenerate spectrum. 

Using this, we show that a compact riemannian four-manifold with self-dual Weyl curvature is anti-self-dual
if $\lambda_{2}\leq\frac{\lambda_{3}}{4}$ and the type is nonpositive. 
This improves the pinching condition and scalar curvature condition in [15].
In Einstein case with positive scalar curvature, we get a result with an improved pinching condition. 
This is possible due to Derdziński's result [8] such that a compact Einstein four-manifold with degenerate spectrum of $W^{+}$ is either anti-self-dual
or $W^{+}$ is nowhere zero. 
In K\"ahler-Einstein surfaces, 
we classify compact K\"ahler-Einstein surfaces with  
\[H_{av}-H_{min}\leq \frac{5}{9}(H_{max}-H_{min}).\]
This improves the coefficient of Guan's result [10].

\vspace{20pt}

\section{\large\textbf{On the holomorphic sectional curvature of Einstein four-manifolds}}\label{S:Intro}
In this section, we consider the maximum and the minimum of the holomorphic sectional curvature of almost-Hermitian Einstein four-manifolds. 
The following result was shown in [12]. 

\newtheorem{Lemma}{Lemma}
\begin{Lemma}
Let $(M, g, \omega, J)$ be an almost-Hermitian four-manifold. 
Then given an anti-self-dual 2-form $\phi$ with unit length , 
there exists an orthonormal basis $\{X, JX, Y, JY\}$ such that $\phi=\frac{X\wedge JX-Y\wedge JY}{\sqrt{2}}$.

\end{Lemma}
\begin{proof}
We follow the proof given in [12]. 
We note that $\frac{\omega}{\sqrt{2}}$ is a self-dual 2-form of unit length. 
Using the correspondence of the Grassmannian
\[Gr^{+}_{2}(T_{p}(M)=\{(\alpha, \beta)\in\Lambda^{+}\oplus\Lambda^{-}:|\alpha|=|\beta|=1\},\]
for a given anti-self-dual 2-form $\phi$ of unit length, 
$\frac{\omega}{\sqrt{2}}+\phi$ is decomposable. 
Then there exist orthonormal 1-forms $e_{1}, e_{2}$ such that 
\[\frac{\omega}{\sqrt{2}}+\phi=e_{1}\wedge e_{2}.\]
Then we get 
\[\frac{\omega}{\sqrt{2}}=\frac{e_{1}\wedge e_{2}+*(e_{1}\wedge e_{2})}{\sqrt{2}}.\]
Since $\omega(X, Y)=g(JX, Y)$, we get $e_{2}=J(e_{1})$. 
Similarly, we have $e_{4}=Je_{3}$ for a positively oriented basis $\{e_{1}, e_{2}, e_{3}, e_{4}\}$.
From this, it follows that 
\[\phi=\frac{e_{1}\wedge Je_{1}-e_{3}\wedge Je_{3}}{\sqrt{2}}.\]
\end{proof}

\begin{Lemma}
Let $(M, g, ,\omega, J)$ be an almost-Hermitian four-manifold and $\lambda_{i}$ be the eigenvalues of $W^{-}$ such that 
$\lambda_{1}\leq\lambda_{2}\leq\lambda_{3}$. 
Then we have 
\[H_{omax}=\frac{s+3s^{*}}{24}+\frac{\lambda_{3}}{2}, \hspace{10pt} H_{omin}=\frac{s+3s^{*}}{24}+\frac{\lambda_{1}}{2},\]
where $H_{omax}(H_{omin})$ be the maximum(minimum) of the orthogonal holomorphic sectional curvature and $s^{*}=2R(\omega, \omega)$. 
\end{Lemma}
\begin{proof}
Let $r^{*}(X, Y):=tr(Z\to R(X, JZ)JY)$. Then $s^{*}:=tr(r^{*})$
and $s^{*}=2R_{1212}+2R_{3434}+4R_{1234}$, where $R_{ijkl}=R(e_{i}\wedge e_{j}, e_{k}\wedge e_{l})$ for an orthonormal basis
$\{e_{1}, e_{2}, e_{3}, e_{4}\}$. 
Suppose $\phi$ is an eigenvector of $W^{-}$ with the eigenvalue $\lambda_{1}$. 
Then there exists an orthonormal basis $\{X, JX, Y, JY\}$ such that $\phi=\frac{1}{\sqrt{2}}\left(X\wedge JX-Y\wedge JY\right)$ by Lemma 1.
Let $X=e_{1}, JX=e_{2}, Y=e_{3}, JY=e_{4}$.
Then we have 
\[\left(\frac{s}{3}-2W^{-}\right)(\phi, \phi)=R_{1313}+R_{1414}+R_{2323}+R_{2424}+2R_{1234}.\]
\[=\frac{s}{2}+\frac{s^{*}}{2}-2(R_{1212}+R_{3434}).\]
It follows that 
 \[\frac{s}{2}+\frac{s^{*}}{2}-4H_{omin}\geq\frac{s}{2}+\frac{s^{*}}{2}-2(R_{1212}+R_{3434})=\frac{s}{3}-2\lambda_{1}.\]
 From this, we get
 \[\frac{s}{6}+\frac{s^{*}}{2}+2\lambda_{1}\geq 4H_{omin}.\]
Let $H_{omin}=\frac{1}{2}(H(X, JX)+H(Y, JY))$ for an orthonormal basis 
$\{X, JX, Y, JY\}$
and let $\phi=\frac{1}{\sqrt{2}}\left(X\wedge JX-Y\wedge JY\right)$. 
Then we have 
\[\frac{s}{3}-2\lambda_{1}\geq\left(\frac{s}{3}-2W^{-}\right)(\phi, \phi)=\frac{s}{2}+\frac{s^{*}}{2}-2(H(X, JX)+H(Y, JY)).\]
From this, we get
\[H_{omin}\geq\frac{s+3s^{*}}{24}+\frac{\lambda_{1}}{2}.\]
It follows that 
$H_{omin}=\frac{s+3s^{*}}{24}+\frac{\lambda_{1}}{2}$. Similarly, we have $H_{omax}=\frac{s+3s^{*}}{24}+\frac{\lambda_{3}}{2}$.
\end{proof}

We note that for an Einstein manifold and an orthonormal basis $\{X, JX, Y, JY\}$, we have $H_{ohol}(X, Y)=H(X)=H(Y)$
since $H(X, JX)=H(Y, JY)$. 
Thus, we get the following result. 

\newtheorem{Corollary}{Corollary}
\begin{Corollary}
Let $(M, g, \omega, J)$ be an almost-Hermitian Einstein four-manifold and $\lambda_{i}$ be the eigenvalues of $W^{-}$ such that 
$\lambda_{1}\leq\lambda_{2}\leq\lambda_{3}$. 
Then we have 
\[H_{max}=\frac{s+3s^{*}}{24}+\frac{\lambda_{3}}{2}, \hspace{10pt} H_{min}=\frac{s+3s^{*}}{24}+\frac{\lambda_{1}}{2},\]
where $H_{max}(H_{min})$ is the maximum(minimum) of the holomorphic sectional curvature 
and $s^{*}=2R(\omega, \omega)$. 
\end{Corollary}
Let $(M, g, \omega)$ be an almost-K\"ahler four-manifold. 
We have 
\[s^{*}=2R(\omega, \omega)=2\left(\frac{s}{12}I+W^{+}\right)(\omega, \omega)=\frac{s}{3}+2W^{+}(\omega, \omega).\]
From the Weitzenb\"ock formula
\[\left<\Delta\omega, \omega\right>=\left<\nabla^{*}\nabla\omega, \omega\right>-2W^{+}(\omega, \omega)+\frac{s}{3}\left<\omega, \omega\right>,\]
we get 
\[0=|\nabla\omega|^{2}-2W^{+}(\omega, \omega)+\frac{2s}{3}\]
since $\omega$ is a harmonic 2-form of length $\sqrt{2}$. 
It follows that 
\[s^{*}=s+|\nabla\omega|^{2}.\]
For a K\"ahler surface, we have $\nabla\omega=0$, and therefore $s^{*}=s$. 
Thus, we get the following result. 
\begin{Corollary}
Let $(M, g, \omega)$ be a K\"ahler-Einstein surface. 
Let $\lambda_{i}$ be the eigenvalues of $W^{-}$ such that $\lambda_{1}\leq\lambda_{2}\leq\lambda_{3}$
and $H_{max}(H_{min})$ be the maximum(minimum) of the holomorphic sectional curvature. 
Then we have 
\[H_{max}=\left(\frac{s}{6}\right)+\frac{\lambda_{3}}{2}, \hspace{5pt} H_{max}=\left(\frac{s}{6}\right)+\frac{\lambda_{1}}{2}.\]
\end{Corollary}

Let $(M, g, \omega)$ be a K\"ahler manifold with complex-dimension $N$ and $H$ be the holomorphic sectional curvature. 
Berger showed
$H\leq\frac{s}{N(N+1)}$ if and only if $H=\frac{s}{N(N+1)}$ [5]. 
Using Lemma 2, this result can be generalized in almost-Hermitian four-manifold with respect to the orthogonal holomorphic sectional curvature. 

\newtheorem{Proposition}{Proposition}
\begin{Proposition}
Let $(M, g, \omega)$ be an almost-Hermitian four-manifold and let $H_{ohol}$ be the orthogonal holomorphic sectional curvature. 
Suppose $H_{ohol}\leq\frac{s+3s^{*}}{24}$ or $H_{ohol}\geq\frac{s+3s^{*}}{24}$. 
Then $H_{ohol}=\frac{s+3s^{*}}{24}$ and $(M, g)$ is self-dual.
\end{Proposition}
\begin{proof}
If $H_{ohol}\leq\frac{s+3s^{*}}{24}$, we get 
\[\frac{s+3s^{*}}{24}+\frac{\lambda_{3}}{2}\leq\frac{s+3s^{*}}{24}.\]
Thus, $\lambda_{3}=0$. Since $\lambda_{1}\leq\lambda_{2}\leq\lambda_{3}$ and $\lambda_{1}+\lambda_{2}+\lambda_{3}=0$, 
we get $\lambda_{1}=\lambda_{2}=\lambda_{3}=0$.
Since
\[H_{omin}=\frac{s+3s^{*}}{24}+\frac{\lambda_{1}}{2},\]
we get $H_{omin}=\frac{s+3s^{*}}{24}$. Thus, $H_{ohol}=\frac{s+3s^{*}}{24}$. 
Similarly, we get $H_{ohol}=\frac{s+3s^{*}}{24}$ if $H_{ohol}\geq\frac{s+3s^{*}}{24}$.
\end{proof}

\begin{Proposition}
Let $(M, g, J)$ be an almost-Hermitian four-manifold with pointwise constant orthogonal holomorphic sectional curvature. 
Then $(M, g)$ is self-dual. 
\end{Proposition}

In [14], it was shown that a compact almost-K\"ahler four-manifold with constant holomorphic sectional curvature with $J$-invariant ricci tensor
is K\"ahler with constant holomorphic sectional curvature.
This result can be extended regarding the orthogonal holomorphic sectional curvature. 
The same proof gives following results. 

\begin{Proposition}
Let $(M, g, \omega, J)$ be a compact almost-K\"ahler four-manifold with pointwise constant orthogonal holomorphic sectional curvature. 
Suppose the ricci-tensor is $J$-invariant. 
Then 
\begin{itemize}
\item $(M, g)$ is Einstein; or
\item $(M, g, \omega)$ is K\"ahler and $(M, g)$ is locally a product space of 2-dimensional riemannian manifolds of constant curvature $K$ and $-K$ with $K\neq 0$. 
\end{itemize}

\end{Proposition}
\begin{Proposition}
Let $(M, g, \omega, J)$ be a compact almost-K\"ahler four-manifold with $J$-invariant ricci-tensor. 
If the orthogonal holomorpohic sectional curvature is constant, then
\begin{itemize}
\item $(M, g, \omega)$ is K\"ahler with constant holomorphic sectional curvature; or
\item $(M, g, \omega)$ is K\"ahler and $(M, g)$ is locally a product space of 2-dimensional riemannian manifolds of constant curvature $K$ and $-K$ with $K\neq 0$. 
\end{itemize}
\end{Proposition}

\vspace{20pt}

\section{\large\textbf{Pinched and harmonic self-dual Weyl curvature}}\label{S:Intro}
Let $(M, g)$ be a compact Einstein manifold with a degenerate spectrum of $W^{+}$. 
Derdziński showed that $(M, g)$ or its double cover is conformal to a K\"ahler metric if $(M, g)$ is not anti-self-dual [8]. 
Wu showed the same result by assuming $det W^{+}>0$ [27]. 
LeBrun proved this result by a different method [20]. Moreover, LeBrun showed that the hypothesis $det W^{+}>0$ can be weakened by 
$\lambda_{2}\leq\frac{\lambda_{3}}{4}$ and $W^{+}$ is nowhere zero where $\lambda_{i}$ are the eigenvalues of $W^{+}$
such that $\lambda_{1}\leq\lambda_{2}\leq\lambda_{3}$. 

\newtheorem{Theorem}{Theorem}
\begin{Theorem}
(LeBrun [20]) Let $(M, g)$ be a compact oriented four-manifold with $\delta W^{+}=0$. 
Suppose $W^{+}$ is nowhere zero and $\lambda_{2}\leq \frac{\lambda_{3}}{4}$,
where $\lambda_{i}$ are the eigenvalues of $W^{+}$ such that $\lambda_{1}\leq\lambda_{2}\leq\lambda_{3}$. 
Then $det W^{+}>0$. Moreover,
\begin{itemize}
\item $b_{+}(M)=1$ and $(M, g)$ is conformal to a K\"ahler metric with positive scalar curvature; or
\item $b_{+}(M)=0$ and the double cover of $(M, g)$ is conformal to a K\"ahler metric with positive scalar curvature. 
\end{itemize}
\end{Theorem}

\begin{Theorem}
(B\"ar, Polombo) Let $(M, g)$ be a compact oriented, riemannian four-manifold with harmonic self-dual Weyl curvature. 
If $(M, g)$ is not anti-self-dual, then the zero set of $W^{+}$ has codimension $\geq 2$.
\end{Theorem}

Let $(M, g)$ be a compact oriented riemannian four-manifold with harmonic self-dual Weyl curvature. 
Suppose $\lambda_{2}\leq\frac{\lambda_{3}}{4}$, 
where  $\lambda_{i}$ are the eigenvalues of $W^{+}$
such that $\lambda_{1}\leq\lambda_{2}\leq\lambda_{3}$. 
Then $\lambda_{3}$ is a well-defined smooth function on $X:=M-N$, 
where $N$ is the zero set of $W^{+}$. 
If not, then $\lambda_{2}=\lambda_{3}$ at a point $p\in X$.
Since $\lambda_{2}\leq\frac{\lambda_{3}}{4}$, we have $\lambda_{2}=\lambda_{3}=0$ at $p$ 
and therefore $p\in N$ since $\lambda_{1}+\lambda_{2}+\lambda_{3}=0$, which is a contradiction. 
Therefore, $\lambda_{2}<\lambda_{3}$ on $X$
and $\lambda_{3}$ has multiplicity one and $\lambda_{3}$ is a smooth function on $X$. 
We note that if $det W^{+}\geq 0$, then $\lambda_{2}\leq\frac{\lambda_{3}}{4}$.
Let $\alpha_{h}:=\lambda_{3}$. The eigenspace of $\alpha_{h}$ varies smoothly from a point to point on $X$.
Thus, the eigenspace of $\alpha_{h}$ defines a line bundle on $X$
and this line bundle is extended to $M$ since it is defined outside the set of codimension $\geq 2$.

\begin{Proposition}
Let $(M, h)$ be a compact oriented riemannian four-manifold such that $\delta W^{+}=0$ and 
let $N$ be the zero set of $W^{+}$. 
Let $\lambda_{i}$ be the eigenvalues of $W^{+}$ such that 
$\lambda_{1}\leq\lambda_{2}\leq\lambda_{3}$. Suppose $(M, h)$ is not anti-self-dual and $\lambda_{2}\leq\frac{\lambda_{3}}{4}$.
Let $\alpha_{h}:=\lambda_3$ on $X:=M-N$. 
Suppose the extended line bundle $L$ of the top eigenspace of $W^{+}$ on $M$ is trivial. 
Let $g:=f^{-2}h$, where $f=\alpha_{h}^{-1/3}$.
Then the eigenvector $\omega$ of $W^{+}_{g}$ with eigenvalue $\alpha_{g}=f^{2}\alpha_{h}$ is defined on $X$
and we have 
$W^{+}_{g}(\omega)=\alpha_{g}\omega$ and $|\omega|_{g}^{2}=2$.
Moreover, $d\omega=0$ on $X$. 
 \end{Proposition}
\begin{proof}
Using Lemma 3, Lemma 4 and $|W^{+}|_{g}^{2}\geq\frac{3}{2}\alpha_{g}^{2}$ since $W^{+}$ is trace-free, we have 
\begin{align*}
&0=\left<\delta_{g}(fW^{+}_{g}), \omega\otimes\omega\right>=\left<\nabla^{*}\nabla fW^{+}+\frac{s}{2}fW^{+}-6fW^{+}\circ W^{+}+2f|W^{+}|^{2}I, \omega\otimes\omega\right>\\
&=\left<\nabla^{*}\nabla fW^{+}, \omega\otimes\omega\right>+\left(\frac{s}{2}W^{+}(\omega, \omega)-6|W^{+}(\omega)|^{2}+2|W^{+}|^{2}|\omega|^{2}\right)f.\\
&=2|\nabla\omega|^{2}-2fW^{+}(\nabla^{a}\omega, \nabla_{a}\omega)
+\left(\frac{s}{2}\alpha_{g}|\omega|^{2}-6\alpha_{g}^{2}|\omega|^{2}+2|W^{+}|^{2}|\omega|^{2}\right)fd\mu_{g}\\
&\geq 2|\nabla\omega|^{2}-2f\beta|\nabla\omega|^{2}+\left(\frac{s}{2}\alpha_{g}|\omega|^{2}-3\alpha_{g}^{2}|\omega|^{2}\right)f\\
&\geq2|\nabla\omega|^{2}-2f\frac{\alpha_{g}}{4}|\nabla\omega|^{2}+\left(\frac{s}{2}|\omega|^{2}-3\alpha_{g}|\omega|^{2}\right)\alpha_{g} f\\
&=\frac{3}{2}\left<\omega, \nabla^{*}\nabla\omega\right>+\frac{s}{2}|\omega|^{2}-3W^{+}(\omega, \omega)\\
&=\frac{3}{2}\left(\left<\omega, \nabla^{*}\nabla\omega\right>-2W^{+}(\omega, \omega)+\frac{s}{3}|\omega|^{2}\right)=\frac{3}{2}\left<\omega, \Delta\omega\right>.\\
\end{align*}
Here we used 
\begin{align*}
&\left<\nabla^{*}\nabla(fW^{+}), \omega\otimes\omega\right>\\
&=-2(f\alpha_{g})\left<\omega, \nabla^{a}\nabla_{a}\omega\right>-2fW^{+}(\nabla^{a}\omega, \nabla_{a}\omega)\\
&=\Delta|\omega|^{2}+2|\nabla\omega|^{2}-2fW^{+}(\nabla^{a}\omega, \nabla_{a}\omega)\\
&=2|\nabla\omega|^{2}-2fW^{+}(\nabla^{a}\omega, \nabla_{a}\omega).\\
\end{align*}
Then, we have
\[0\geq \int_{X}\left<\omega, \Delta\omega\right>d\mu_{g}=\int_{X}\left<\omega, (d\delta+\delta d)\omega\right>d\mu_{g},\] 
where $\delta=-*d*$.
Let $X_{\varepsilon}:=\{x\in M|dist_{h}(x, N)\geq\varepsilon>0\}$.
We have 
\[\int_{X_{\varepsilon}}(\delta d\omega, \omega)d\mu=-\int_{X_{\varepsilon}}(*d*d\omega, \omega)d\mu\]
\[=-\int_{X_{\varepsilon}}(d*d\omega)\wedge\omega
=-\int_{X_{\varepsilon}}d(*d\omega\wedge\omega)-\int_{X_{\varepsilon}}*d\omega\wedge d\omega\]
\[=-\int_{\partial X_{\varepsilon}}*d\omega\wedge\omega+\int_{X_{\varepsilon}}|d\omega|^{2}d\mu_{g}.\]
When a form $\alpha$ on $X_{\varepsilon}$ is integrated on $\partial X_{\varepsilon}$, by $\alpha$, we mean $\alpha|_{\partial{X_{\varepsilon}}}$. 
Similarly, 
\[(d\delta\omega, \omega)=-(d(*d*)\omega, \omega)=-(\omega, d(*d*)\omega)=-\int_{X_{\varepsilon}}\omega\wedge d(*d*)\omega\]
\[=-\int_{X_{\varepsilon}}d(\omega\wedge(*d*)\omega)+\int_{X_{\varepsilon}}d\omega\wedge(*d*)\omega
=-\int_{\partial X_{\varepsilon}}\omega\wedge *d\omega+\int_{X_{\varepsilon}}d\omega\wedge *d\omega\]
\[=-\int_{\partial X_{\varepsilon}}\omega\wedge*d\omega-\int_{X_{\varepsilon}}*d\omega\wedge d\omega=-\int_{\partial X_{\varepsilon}}\omega\wedge*d\omega+\int_{X_{\varepsilon}}|d\omega|^{2}d\mu_{g}.\]
From this, it follows that 
\[0\geq\int_{X_{\varepsilon}}\left<(d\delta_{g}+\delta_{g} d)\omega, \omega\right>d\mu_{g}=2\int_{X_{\varepsilon}}|d\omega|^{2}d\mu_{g}
-2\int_{\partial X_{\varepsilon}}\omega\wedge *d\omega.\]
Then, we have
\[\int_{\partial X_{\varepsilon}}\omega\wedge *d\omega\geq\int_{X_{\varepsilon}}|d\omega|_{g}^{2}d\mu_{g}\]

Let $\tilde{g}=g|_{\partial X_{\varepsilon}}$ and $\tilde{h}=h|_{\partial X_{\varepsilon}}$. 
Let $\alpha':=\alpha|_{\partial X_{\varepsilon}}$ for a differential form $\alpha$. 
By $|\alpha|_{\tilde{g}}$, we mean 
$|\alpha|_{\tilde{g}}=|\alpha'|_{\tilde{g}}$.
Then we note that $|\alpha|_{g}\geq|\alpha'|_{\tilde{g}}$.
Let $\varepsilon>0$ be a sufficiently small positive number. Then for fixed such $\varepsilon$, 
\[\int_{\partial X_{\varepsilon}}|*_{g}d\omega|^{2}_{\tilde{g}}d\mu_{\tilde{g}}\leq C'({\varepsilon}),\]
where $C'({\varepsilon})$ is a positive number which depends on $\varepsilon$. 
We note that 
\[|\omega|_{\tilde{g}}\leq |\omega|_{g}, \hspace{5pt} |*_{g}d\omega|_{\tilde{g}}\leq|*_{g}d\omega|_{g}.\]
and
 \[|*_{\tilde{g}}\omega|_{\tilde{g}}=|\omega|_{\tilde{g}}, \hspace{5pt} |d\omega|^{2}_{g}=|*_{g}d\omega|_{g}^{2}.\]
Then by the Cauchy-Schwarz inequality, we have
\[\left|\int_{\partial X_{\varepsilon}}\omega\wedge*d\omega\right|
=\left|\int_{\partial X_{\varepsilon}}\left<*_{\tilde{g}}\omega,  *_{g}d\omega\right>_{\tilde{g}}d\mu_{\tilde{g}}\right|\]
\[\leq\left(\int_{\partial X_{\varepsilon}}|*_{\tilde{g}}\omega|^{2}_{\tilde{g}}d\mu_{\tilde{g}}\right)^{1/2}
\left(\int_{\partial X_{\varepsilon}}|*_{g}d\omega|^{2}_{\tilde{g}}d\mu_{\tilde{g}}\right)^{1/2}\]
\[=\left(\int_{\partial X_{\varepsilon}}|\omega|^{2}_{\tilde{g}}d\mu_{\tilde{g}}\right)^{1/2}
\left(\int_{\partial X_{\varepsilon}}|*_{g}d\omega|^{2}_{\tilde{g}}d\mu_{\tilde{g}}\right)^{1/2}\]
Then we have 
\[\left(\int_{\partial X_{\varepsilon}}|\omega|^{2}_{\tilde{g}}d\mu_{\tilde{g}}\right)^{1/2}\left(\int_{\partial X_{\varepsilon}}|*_{g}d\omega|^{2}_{\tilde{g}}d\mu_{\tilde{g}}\right)^{1/2}
\geq\left|\int_{\partial X_{\varepsilon}}\omega\wedge*d\omega\right|\]
\[\geq\int_{X_{\varepsilon}}|d\omega|_{g}^{2}d\mu_{g}=\int_{X_{\varepsilon}}|*_{g}d\omega|_{g}^{2}d\mu_{g}
\geq\int_{\partial X_{\varepsilon}}|*_{g}d\omega|^{2}_{\tilde{g}}d\mu_{\tilde{g}}\]
Since $\alpha_{h}$ goes to 0 near $N$, $\alpha_{h}$ is bounded on $\partial X_{\varepsilon}$.
Since $N$ is codimension $\geq 2$ by Theorem 2, we note that $Vol(\partial X_{\varepsilon}, \tilde{h})$ is $\mathcal{O}({\varepsilon})$. 
Then we have 
\[\int_{\partial X_{\varepsilon}}|\omega|^{2}_{\tilde{g}}d\mu_{\tilde{g}}\leq2\int_{\partial X_{\varepsilon}}f^{-3}d\mu_{\tilde{h}}=2\int_{\partial X_{\varepsilon}}\alpha_{h}d\mu_{\tilde{h}}
\leq C\varepsilon,\]
for a positive constant $C$.
Thus, for such $\varepsilon$, we get 
\[C\varepsilon\geq\int_{\partial X_{\varepsilon}}|*_{g}d\omega|^{2}_{\tilde{g}}d\mu_{\tilde{g}}.\]
Then we have 
\[C{\varepsilon}\geq\left(\int_{\partial X_{\varepsilon}}|\omega|^{2}_{\tilde{g}}d\mu_{\tilde{g}}\right)^{1/2}
\left(\int_{\partial X_{\varepsilon}}|*_{g}d\omega|^{2}_{\tilde{g}}d\mu_{\tilde{g}}\right)^{1/2}
\geq\int_{X_{\varepsilon}}|d\omega|_{g}^{2}d\mu_{g}.\]
By letting $\varepsilon\to 0$, we get $d\omega=0$ on $X$. 
 \end{proof}

\begin{Lemma}
Let $(M, g)$ be an oriented riemannian four-manifold and let $\omega$ be a smooth self-dual 2-form defined on $M$ such that $|\omega|_{g}$ is constant. 
Suppose $\omega$ is an eigenvector of $W^{+}$ with the largest eigenvalue.
If $\beta=\beta_{g}$ is the middle eigenvalue of $W_{g}^{+}$, then 
\[W_{g}^{+}(\nabla_{e}\omega, \nabla^{e}\omega)\leq \beta|\nabla\omega|^{2}.\]
\end{Lemma}
\begin{proof}
Since $|\omega|_{g}^{2}=2$, $\nabla\omega\in \Lambda^{1}\otimes\omega^{\perp}$. 
\end{proof}

\begin{Lemma}
(Biquard-Gauduchon-LeBrun)
Let $(M, g)$ be an oriented riemannian four-manifold. 
Suppose $\omega$ is a self-dual harmonic 2-form with contant length such that 
$W^{+}=\alpha_{g}\omega$, where $\alpha_{g}>0$. Let $f:=\alpha_{g}^{-1}$. Then 
\[\left<\nabla^{*}\nabla fW^{+}, \omega\otimes\omega\right>=-2(f\alpha_{g})(\omega, \nabla^{a}\nabla_{a}\omega)-2fW^{+}(\nabla^{a}\nabla_{a}\omega).\]
\end{Lemma}
\begin{Theorem}
Let $(M, h)$ be a compact oriented four-manifold with harmonic self-dual Weyl curvature. 
Suppose $\lambda_{2}\leq\frac{\lambda_{3}}{4}$, where $\lambda_{i}$ are the eigenvalues of $W^{+}$ such that $\lambda_{1}\leq\lambda_{2}\leq\lambda_{3}$. 
Then $W^{+}$ has a degenerate spectrum and $det W^{+}\geq 0$.
\end{Theorem}
\begin{proof}
If $(M, h)$ is anti-self-dual, $W^{+}$ is degenerate. 
 Suppose $(M, h)$ is not anti-self-dual and the extended line bundle $L$ of the top eigenspace of $\alpha_{h}$ on $M$ is trivial.  
 Let $\alpha_{h}:=\lambda_{3}$ and $g:=f^{-2}h$ on $X:=M-N$, where $N$ is the zero set of $W^{+}_{h}$ and $f=\alpha_{h}^{-1/3}$.
 Let $\omega$ be the eigenvector of $W^{+}_{g}$ with the eigenvalue $\alpha_{g}=f^{2}\alpha_{h}$.
 Then $\omega$ is a self-dual 2-form such that $W^{+}_{g}(\omega)=\alpha_{g}\omega, |\omega|_{g}=\sqrt{2}$.
 Then by Proposition 5, $d\omega=0$ on $X$. 
 From this, we have
 \[|W^{+}(\omega)|^{2}=\frac{1}{2}\left[W^{+}(\omega, \omega)\right]^{2}, \hspace{5pt} \alpha_{g}=\frac{1}{2}W^{+}(\omega, \omega).\]

 Since $\delta_{h} W^{+}_{h}=0$, we have $\delta_{g}(fW^{+}_{g})=0$ [18] and therefore, 
 \[0=\left<\delta_{g}(fW^{+}_{g}), \omega\otimes\omega\right>
 =\left<\nabla^{*}\nabla fW^{+}, \omega\otimes\omega\right>+\left(\frac{s}{2}W^{+}(\omega, \omega)-6|W^{+}(\omega)|^{2}+2|W^{+}|^{2}|\omega|^{2}\right)f,\]  
 with respect to $g$. 
 Then we have 
 \begin{align*}
&\left<\nabla^{*}\nabla(fW^{+}), \omega\otimes\omega\right>\\
&=-2(f\alpha_{g})\left<\omega, \nabla^{a}\nabla_{a}\omega\right>-2fW^{+}(\nabla^{a}\omega, \nabla_{a}\omega)\\
&=\left(W^{+}(\omega, \omega)|\nabla\omega|^{2}+\frac{1}{2}W^{+}(\omega, \omega)|\nabla\omega|^{2}\right)f\\
&=\frac{3}{2}W^{+}(\omega, \omega)|\nabla\omega|^{2}f.\\
\end{align*}
Here we used 
\[W^{+}(\nabla_{e}\omega, \nabla^{e}\omega)=-\frac{1}{4}|\nabla\omega|^{2}W^{+}(\omega, \omega)\]
from Lemma 1 in [18] for a four-dimensional almost-K\"ahler structure. 
Since $\delta=-*d*$, $\Delta\omega=0$. 
From the Weitzenb\"ock formula for a self-dual 2-form, 
\[\Delta\omega=\nabla^{*}\nabla\omega-2W^{+}(\omega)+\frac{s}{3}\omega,\]
 it follows that 
 \[\nabla^{*}\nabla\omega=2W^{+}(\omega)-\frac{s}{3}\omega.\]
 Since $|\omega|=\sqrt{2}$, using $-\frac{1}{2}\Delta|\omega|^{2}=\left<\nabla\omega, \nabla\omega\right>-\left<\nabla^{*}\nabla\omega, \omega\right>$, we have 
 \[|\nabla\omega|^{2}=2W^{+}(\omega, \omega)-\frac{2}{3}s.\]
 
Using
\[|W^{+}|^{2}\geq\frac{3}{2}\alpha^{2}_{g}=\frac{3}{8}(W^{+}(\omega, \omega))^{2},\]
and 
\[\frac{s}{2}=\frac{3}{2}W^{+}(\omega, \omega)-\frac{3}{4}|\nabla\omega|^{2},\]
on X, we have 
\begin{align*}
&0=\left(\frac{3}{2}W^{+}(\omega, \omega)|\nabla\omega|^{2}+\frac{s}{2}W^{+}(\omega, \omega)-6|W^{+}(\omega)|^{2}+2|W^{+}|^{2}|\omega|^{2}\right)f\\
&\geq\left(\frac{3}{2}W^{+}(\omega, \omega)|\nabla\omega|^{2}+\left[\frac{3}{2}W^{+}(\omega, \omega)-\frac{3}{4}|\nabla\omega|^{2}\right]W^{+}(\omega, \omega)
-\frac{3}{2}[W^{+}(\omega, \omega])]^{2}\right)f\\
&=\left(\frac{3}{4}W^{+}(\omega, \omega)|\nabla\omega|^{2}\right)f.
\end{align*}
 
Since $W^{+}(\omega, \omega)>0$ on $X$, 
$\nabla\omega=0$ on $X$.
Thus, $(X, g, \omega)$ is K\"ahler.
Then, $W^{+}_{g}$ is degenerate on $X$ and therefore, $W^{+}_{h}$ is degenerate on $X$. 
Since $W^{+}_{h}=0$ on $M-X$, we get $W^{+}_{h}$ is degenerate on $M$. 

Suppose the extended line bundle $L$ on $M$ is not trivial. 
Let $(\hat{M}, \hat{h})$ be the double cover of $M$ with the pull-back metric $\hat{h}$ of $h$. 
Then on $\hat{X}$, which is the pull-back of $X$, there exists the eigenvector of $W^{+}_{\hat{g}}$ with the eigenvalue $\alpha_{\hat{g}}$, 
where $\hat{g}$ is the pull-back metric of $g$ on $\hat{X}$. 
By the argument above, $(\hat{M}, \hat{h})$ has a degenerate spectrum of $W^{+}_{\hat{h}}$.
Therefore, $W^{+}$ has a degenerate spectrum. 
\end{proof}

\begin{Theorem}
(LeBrun [18]) Let $(M, g, \omega)$ be a compact connected almost-K\"ahler four-manifold with $W^{+}(\omega, \omega)\geq 0$. 
Suppose there exists $f>0$ such that $h:=f^{2}g$ has $\delta W^{+}=0$. 
Then either $g$ is K\"ahler with scalar curvature $s=c/f$ for some positive constant $c$, or $g$ is anti-self-dual. 
\end{Theorem}

LeBrun also showed that if $s^{*}\geq 0$, then $W^{+}(\omega, \omega)\geq 0$ [18].
If it is required that a certain conformal class with an assigned conformal factor has $\delta W^{+}=0$
and $\omega$ is an eigenvector of $W^{+}$ such that $W^{+}(\omega, \omega)\geq 0$, 
the proof of Theorem 3 shows that we get a result for a connected
almost-K\"ahler four-manifold which is not necessarily compact. 

\begin{Proposition}
Let $(M, g, \omega)$ be a connected almost-K\"ahler four-manifold. 
Suppose $W^{+}(\omega)=\alpha_{g}\omega$, where $\alpha_{g}>0$ and $\delta(fW^{+})=0$, where $f=\alpha^{-1}$. 
Then either $(M, g)$ is anti-self-dual or $(M, g, \omega)$ is K\"ahler with scalar curvature $s=c/f$ for a positive constant $c$. 
\end{Proposition}

Let $(M, g)$ be a compact, oriented riemannain four-manifold. Then by [2], [24], the conformal class of $g$
contains a metric of constant scalar curvature.
Moreover, if two metrics with fixed signs of the scalar curvature are conformally equivalent, then signs of the scalar curvature are the same [17]. 
The sign of the constant scalar curvature is called the type of $[g]$. 
Theorem 5 improves the main result in [15]. 

\begin{Theorem}
Let $(M, h)$ be a compact oriented four-manifold with harmonic self-dual Weyl curvature. 
Suppose $\lambda_{2}\leq\frac{\lambda_{3}}{4}$, where $\lambda_{i}$ are the eigenvalues of $W^{+}$ such that $\lambda_{1}\leq\lambda_{2}\leq\lambda_{3}$
and $(M, h)$ is nonpositive type. Then $(M, h)$ is anti-self-dual. 
\end{Theorem}
\begin{proof}
Suppose $(M, h)$ is not anti-self-dual. By Theorem 3, $(M, h)$ is conformal to a K\"ahler metric with positive scalar curvature on $X:=M-N$, 
where $N$ is the nonempty zero set of $W^{+}$. 
Since $(M, h)$ is nonpositive type, there is a metric $h'\in[h]$ such that the scalar curvature of $h'$ is negative or zero. 
Then $h'=(f'')^{2}h$ for a smooth function $f''$ on $M$. Since $h=f^{2}g$, where $f=\alpha_{h}^{-1/3}$, 
we have $h'=(f''f)^{2}g$ on $X$. Thus, $g=(f')^{2}h'$, where $f'=(f''f)^{-1}$ is a positive smooth function on $X$. 
Then we have 
\[s_{g}f'^{3}=(6\Delta_{h'}+s_{h'})f'.\]
Since $s_{h'}\leq0$ and $s_{g}>0$, we have $\Delta f'>0$ on $(X, h')$. 
By Proposition 7, we get the result. 
\end{proof}

\begin{Proposition}
Let $(M, g)$ be a compact riemannian manifold. 
Suppose $N$ is a nonempty set with codimension $\geq 2$. 
Then there does not exist a positive strictly superharmonic function on $X:=M-N$.
\end{Proposition}
\begin{proof}
 Let $X_{\varepsilon}=\{x\in M|dist_{g}(x, N)\leq\varepsilon\}$.
Suppose there is a strictly superharmonic function $\Phi$ on $X$.
Then we have 
\[\int_{X_{\varepsilon}}\Phi^{-1}\Delta\Phi d\mu=\int_{X_{\varepsilon}}\left<d\Phi^{-1}, d\Phi\right>d\mu
-\int_{\partial X_{\varepsilon}}\Phi^{-1}\frac{\partial\Phi}{\partial\nu}da>0.\]
From this, we get 
\[-\int_{\partial X_{\varepsilon}}\Phi^{-1}\frac{\partial\Phi}{\partial\nu}>\int_{X_{\varepsilon}}\Phi^{-2}|d\Phi|^{2}d\mu.\]
Fix $\varepsilon>0$. 
Then we have 
\[\int_{\partial X_{\varepsilon}}\Phi^{-2}\left(\frac{\partial\Phi}{\partial\nu}\right)^{2}da\leq C_{\varepsilon},\]
where $C_{\varepsilon}$ is a positive constant which depends on $\varepsilon$. 
We note that $\int_{\partial X_{\varepsilon}}da=\mathcal{O}(\varepsilon)$ since $N$ is codimension $\geq 2$. In particular, it is bounded. 
By the Cauchy-Schwarz inequality, we get
\[\left(\int_{\partial X_{\varepsilon}}da\right)^{1/2}\left(\int_{\partial X_{\varepsilon}}\Phi^{-2}\left(\frac{\partial\Phi}{\partial\nu}\right)^{2}da\right)^{1/2}
\geq \left|\int_{\partial X_{\varepsilon}}\Phi^{-1}\frac{\partial\Phi}{\partial\nu}da\right|\]
\[>\int_{X_{\varepsilon}}\Phi^{-2}|d\Phi|^{2}d\mu\geq\int_{\partial X_{\varepsilon}}\Phi^{-2}\left(\frac{\partial\Phi}{\partial\nu}\right)^{2}da.\]
Thus, we have 
\[\left(\int_{\partial X_{\varepsilon}}da\right)>\int_{\partial X_{\varepsilon}}\Phi^{-2}\left(\frac{\partial\Phi}{\partial\nu}\right)^{2}da.\]
In particular, $\int_{\partial X_{\varepsilon}}\Phi^{-2}\left(\frac{\partial\Phi}{\partial\nu}\right)^{2}da$ is bounded. 
Then we get
\[0=\lim_{\varepsilon\to 0}\left(\int_{\partial X_{\varepsilon}}da\right)^{1/2}
\lim_{\varepsilon\to 0}\left(\int_{\partial X_{\varepsilon}}\Phi^{-2}\left(\frac{\partial\Phi}{\partial\nu}\right)^{2}da\right)^{1/2}
\geq \lim_{\varepsilon\to 0}\int_{X_{\varepsilon}}\Phi^{-2}|d\Phi|^{2}d\mu\geq 0.\]
Thus, $\Phi^{-2}d\Phi=0$ on $X$. Since $\Phi^{-1}>0$, we get $d\Phi=0$ on $X$, which is a contradiction since $\Phi$ is strictly superharmonic on $X$.

\end{proof}

\begin{Proposition}
Let $(M, g)$ be a compact  riemannian four-manifold. 
Suppose $N$ is a nonempty set with codimension $\geq 2$. 
Then there does not exist a positive bounded strictly subharmonic function on $X:=M-N$. 
\end{Proposition}
\begin{proof}
Let $X_{\varepsilon}=\{x\in M|dist_{g}(x, N)\leq\varepsilon\}$.
Suppose $f$ is a positive bounded subharmonic function on $X$. 
Since $\Delta f<0$ and, 
\[\int_{X_{\varepsilon}}f\Delta fd\mu=\int_{\partial X_{\varepsilon}}f\frac{\partial f}{\partial\nu}da-\int_{X_{\varepsilon}}(df, df)d\mu.\]
we get 
\[\int_{\partial X_{\varepsilon}}f\frac{\partial f}{\partial\nu}da>\int_{X_{\varepsilon}}(df, df)d\mu.\]
Since $f$ is bounded and $N$ is codimension $\geq 2$, we have $\int_{\partial X_{\varepsilon}}|f|^{2}da=\mathcal{O}(\varepsilon)$. 
In particular, $\int_{\partial X_{\varepsilon}}|f|^{2}da\leq C$, where $C$ is a positive constant. 
For fixed $\varepsilon$, 
$\int_{\partial X_{\varepsilon}}\left|\frac{\partial f}{\partial\nu}\right|^{2}da \leq C_{\varepsilon}$, 
where $C_{\varepsilon}$ is a constant which depends on $\varepsilon$. 
Then by the Cauchy-Schwarz inequality, we have 
\[\left(\int_{\partial X_{\varepsilon}}|f|^{2}da\right)^{1/2}\left(\int_{\partial X_{\varepsilon}}\left|\frac{\partial f}{\partial\nu}\right|^{2}da\right)^{1/2}
\geq\left|\int_{\partial X_{\varepsilon}}f\frac{\partial f}{\partial\nu}da\right|\]
Then we have 
\[\left(\int_{\partial X_{\varepsilon}}|f|^{2}da\right)^{1/2}\left(\int_{\partial X_{\varepsilon}}\left|\frac{\partial f}{\partial\nu}\right|^{2}da\right)^{1/2}
\geq\left|\int_{\partial X_{\varepsilon}}f\frac{\partial f}{\partial\nu}da\right|\]
\[>\int_{X_{\varepsilon}}\left(df, df\right)d\mu\geq\int_{\partial X_{\varepsilon}}\left|\frac{\partial f}{\partial\nu}\right|^{2}da.\]
From this, we get
\[\int_{\partial X_{\varepsilon}}\left|\frac{\partial f}{\partial\nu}\right|^{2}da<\int_{\partial X_{\varepsilon}}|f|^{2}da\leq C.\]
Then we get 
\[0=\lim_{\varepsilon\to 0}\left(\int_{\partial X_{\varepsilon}}|f|^{2}da\right)^{1/2}
\lim_{\varepsilon\to 0}\left(\int_{\partial X_{\varepsilon}}\left|\frac{\partial f}{\partial\nu}\right|^{2}da\right)^{1/2}\]
\[\geq\lim_{\varepsilon\to 0}\left|\int_{\partial X_{\varepsilon}}f\frac{\partial f}{\partial\nu}da\right|
\geq \lim_{\varepsilon\to 0}\int_{X_{\varepsilon}}\left(df, df\right)d\mu\geq 0.\]
Thus, $df=0$ on $X$, which is a contradiction since $f$ is strictly subharmonic on $X$. 
\end{proof}

Using Theorem 3, Theorem 16 in [15] can be improved in the following way. 

\begin{Theorem}
Let $(M, g)$ be a compact oriented four-manifold with harmonic self-dual Weyl curvature. 
Suppose $\lambda_{2}\leq\frac{\lambda_{3}}{4}$, where $\lambda_{i}$ are the eigenvalues of $W^{+}$ such that $\lambda_{1}\leq\lambda_{2}\leq\lambda_{3}$. 
If $\lambda_{1}\leq-\frac{s}{12}$ or $\frac{s}{6}\leq\lambda_{3}$, then
\begin{itemize}
\item $(M, g)$ is anti-self-dual; or
\item $b_{+}(M)=1$ and $(M, g)$ is K\"ahler with positive constant scalar curvature; or
\item $b_{+}(M)=0$ and the double cover of $(M, g)$ is K\"ahler metric with positive constant scalar curvature. 
\end{itemize}
\end{Theorem}
\begin{proof}
By Theorem 3, we have 
\[\lambda_{1}=\lambda_{2}=-\frac{\lambda_{3}}{2}.\]
In particular, $\lambda_{1}=\lambda_{2}\leq 0$. 
Then we have 
\[2\lambda_{2}^{2}+4\lambda_{1}\lambda_{3}-\frac{\lambda_{2}s}{2}=-6\lambda_{1}\left(\lambda_{1}+\frac{s}{12}\right)\leq 0.\]
Then by the same argument with Theorem 16 in [15], we get the result. 
\end{proof}

\vspace{20pt}

\section{\large\textbf{Einstein four-manifolds with pinched self-dual Weyl curvature}}\label{S:Intro}

\begin{Theorem}
Let $(M, h)$ be a compact oriented Einstein four-manifold. 
Suppose $\lambda_{2}\leq\frac{\lambda_{3}}{4}$, where $\lambda_{i}$ are the eigenvalues of $W^{+}$ such that $\lambda_{1}\leq\lambda_{2}\leq\lambda_{3}$. 
Suppose $(M, h)$ is not anti-self-dual. Then $W^{+}$ is not zero everywhere. Moreover, 
\begin{itemize}
\item $b_{+}(M)=1$ and $(M, h)$ is conformal to a K\"ahler metric with positive scalar curvature; or
\item $b_{+}(M)=0$ and the pull-back of $(M, h)$ to the double cover of $M$ is conformal to a K\"ahler metric with positive scalar curvature. 
\end{itemize}
\end{Theorem}
\begin{proof}
By Theorem 3, $W^{+}$ has a degenerate spectrum. 
For an oriented Einstein four-manifold with degenerate spectrum of $W^{+}$, Derdziński showed 
that either $(M, h)$ is anti-sekf-dual or $W^{+}$ is nowhere zero (Proposition 5 in [8]).
Then the result follows from Theorem 1. 
\end{proof}

We note that LeBrun showed there are 15 compact, oriented Einstein four-manifolds with $det W^{+}>0$, up to diffeomorphism [21].

\begin{Theorem}
Let $(M, g, \omega, J)$ be a compact K\"ahler-Einstein surface. Suppose 
\[H_{av}-H_{min}\leq\frac{5}{9}(H_{max}-H_{min}),\]
where $H_{max}(H_{min})$ denotes the maximum(minimum) of the holomorphic sectional curvature respectively 
and $H_{av}$ the average of the holomorphic sectional curvature. Then 
\begin{itemize}
\item $(M, g, \omega)$ has constant holomorphic sectional curvature; or
\item $(M, g, \omega)$ is $\mathbb{CP}_{1}\times\mathbb{CP}_{1}$ with the standard metric of constant curvature; or
\end{itemize}
\end{Theorem}
\begin{proof}
By changing the orientation, 
we can use $W^{-}$ instead of $W^{+}$ in Theorem 7. 
Suppose $W^{-}\not\equiv 0$.
By Berger's Theorem [5], we have $H_{av}(p)=\frac{s}{6}$. 
Let $\lambda_{i}$ be the eigenvalues of $W^{-}$ such that $\lambda_{1}\leq\lambda_{2}\leq\lambda_{3}$. 
By Corollary 2, 
it follows that $-\frac{4}{5}\lambda_{1}\leq\lambda_{3}$. We note that $-\frac{4}{5}\lambda_{1}\leq\lambda_{3}$ is equivalent to $\lambda_{2}\leq\frac{\lambda_{3}}{4}$
since $\lambda_{1}+\lambda_{2}+\lambda_{3}=0$. 
Suppose $(M, g, \omega)$ is not self-dual. By Theorem 7, suppose $(\overline{M}, g)$ is an Einstein manifold 
which is conformal to a K\"ahler metric with positive scalar curvature. 
In particular, $\overline{M}$ admits a complex structure $J'$. Since $(M, g, \omega, J)$ is K\"ahler with positive scalar curvature. 
$b_{+}(M)=1$ [4], [28]. 
Thus, $\tau(M)=0$.
If $(\overline{M}, g, J')$ is not K\"ahler, $(\overline{M}, J')$ is $\mathbb{CP}_2{}\#\overline{\mathbb{CP}_{2}}$ by [8], [16]. 
In particular, $M$ is diffeomorphic to $\mathbb{CP}_2{}\#\overline{\mathbb{CP}_{2}}$. 
Since $(M, g, \omega, J)$ is K\"ahler-Einstein, $(M, J)$ is biholomorphic to a Del-Pezzo surface. 
Since $M$ is diffeomorphic to $\mathbb{CP}_2{}\#\overline{\mathbb{CP}_{2}}$, $(M, J)$ is biholomorphic to $\mathbb{CP}_{2}\#\overline{\mathbb{CP}_{2}}$.
On the other hand, $\mathbb{CP}_2{}\#\overline{\mathbb{CP}_{2}}$ does not admit a K\"ahler-Einstein metric [22]. 
Suppose $(\overline{M}, g, J')$ is K\"ahler. 
Then $(M, g, \omega)$ splits and $(M, g, \omega)$ is locally symmetric. 
Thus, $(M, g, \omega)$ is $\mathbb{CP}_{1}\times\mathbb{CP}_{1}$ with the standard metric of constant curvature ([9],  p. 125). 
Suppose the double cover with the opposite orientation $(\hat{\overline{M}}, \hat{g})$ is conformal to a K\"ahler metric with positive scalar curvature. 
Then $b_{-}=0$. By Corollary 5 in [13], we get $(M, g, \omega)$ is self-dual K\"ahler-Einstein with positive scalar curvature, 
which is $\mathbb{CP}_{2}$ with the Fubini-Study metric up to rescaling.
This is a contradiction since $\mathbb{CP}_{2}$ is simply-connected. 
If $(M, g)$ is self-dual, $(M, g, \omega)$ is a K\"ahler surface with constant holomorphic sectional curvature by Corollary 2. 
\end{proof}

\begin{Theorem}
(Hall-Murphy) Let $(M, g, J)$ be an almost Hermitian four-manifold and $H$ be the holomorphic sectional curvature. 
Then 
\[H_{av}(p):=\frac{1}{Vol(\mathbb{S}^{3})}\int_{S_{p}}Hd\sigma=\frac{s+3s^{*}}{24},\]
where $S_{p}$ is the unit sphere in the tangent space at $p$
and $d\sigma$ is the volume for the sphere of radius 1. 
\end{Theorem}

\begin{Theorem}
Let $(M, g, \omega)$ be a compact almost-K\"ahler Einstein four-manifold. Suppose 
\[H_{av}-H_{min}\leq\frac{5}{9}(H_{max}-H_{min}),\]
where $H_{max}(H_{min})$ denotes the maximum(minimum) of the holomorphic sectional curvature respectively 
and $H_{av}$ the average of the holomorphic sectional curvature. Then 
\begin{itemize}
\item $(M, g, \omega)$ is K\"ahler surface with nonnegative constant holomorphic sectional curvature; or
\item $(M, g, \omega)$ is $\mathbb{CP}_{1}\times\mathbb{CP}_{1}$ with the standard metric of constant curvature; or
\item $(M, g, \omega)$ is self-dual almost-K\"ahler Einstein with negative scalar curvature. 
\end{itemize}
\end{Theorem}
\begin{proof}
By Theorem 9, we have $H_{av}(p):=\frac{1}{Vol(\mathbb{S}^{3})}\int_{S_{p}}Hd\sigma=\frac{s+3s^{*}}{24}$.
By Corollary 1, we get $\lambda_{2}\leq\frac{\lambda_{3}}{4}$, where $\lambda_{i}$ are the eigenvalues of $W^{-}$ such that 
$\lambda_{1}\leq\lambda_{2}\leq\lambda_{3}$. Suppose the scalar curvature is nonpositive. Then $(M, g)$ is self-dual by Theorem 7. 
Suppose the scalar curvature is positive and $(M, g)$ is not self-dual. 
We note that a compact almost-K\"ahler Einstein four-manifold with nonnegative scalar curvature is K\"ahler-Einstein [25]. 
Then the result follows from Theorem 8. 
\end{proof}

Goldberg's conjecture is that a compact almost-K\"ahler Einstein manifold is K\"ahler-Einstein. 
In case of nonnegative scalar curvature, this was shown by Sekigawa [25].
When there is a pinching condition on $W^{+}$ instead of $W^{-}$, we get one of the special cases of Goldberg's conjecture.

\begin{Proposition}
Let $(M, g, \omega)$ be a compact almost-K\"ahler Einstein four-manifold. Suppose $\lambda_{2}\leq\frac{\lambda_{3}}{4}$,
where $\lambda_{i}$ are the eigenvalues of $W^{+}$ such that $\lambda_{1}\leq\lambda_{2}\leq\lambda_{3}$. 
\begin{itemize}
\item$(M, g, \omega)$ is K\"ahler-Einstein with positive scalar curvature; or
\item The universal cover of $(M, g, \omega)$ is a torus with the flat metric or $K3$-surface with the Calabi-Yau metric.
\end{itemize}
\end{Proposition}
\begin{proof}
Suppose $(M, g)$ is not anti-self-dual. Then by Theorem 7, $(M, g)$ is conformal to a K\"ahler metric with positive scalar curvature. 
Moreover, $b_{+}=1$. Since a self-dual harmonic 2-form is conformally invariant, $\omega$ is a K\"ahler form with respect to a conformal metric. 
Since the length of an almost-K\"ahler form and a K\"ahler form is constant, $(M, g, \omega)$ is K\"ahler-Einstein with positive scalar curvature. 
Suppose $(M, g, \omega)$ is anti-self-dual almost-K\"ahler Einstein four-manifold. 
Thus, $s\leq 0$. 
Suppose $s=0$. Then the universal cover of $(M, g, \omega)$ is a torus with the flat metric or $K3$-surface with the Calabi-Yau metric. 
Suppose $s<0$. 
Since $s^{*}=\frac{s}{3}+2W^{+}(\omega, \omega)$, we get $s^{*}=\frac{s}{3}$ for an anti-self-dual almost-K\"ahler four-manifold. 
Thus, $s^{*}\neq s$ everywhere. 
On the other hand, for a compact almost-K\"ahler four-manifold, there is at least one point such that $s^{*}=s$
unless $(5\chi+6\tau)(M)=0$ [1]. Thus, we get  $(5\chi+6\tau)(M)=0$.
From the following formula, 
\[\chi(M)=\frac{1}{8\pi^{2}}\int_{M}\left(\frac{s^{2}}{24}+|W^{+}|^{2}+|W^{-}|^{2}-\frac{|ric_{0}|^{2}}{2}\right)d\mu,\]
it follows that 
$\chi(M)\geq 0$. 
On the other hand, for a compact Einstein four-manifold $(2\chi+3\tau)(M)\geq 0$. 
Since $(5\chi+6\tau)(M)=0$, we get $\chi(M)\leq 0$. 
Thus, $\chi(M)=0$. 
Then we have $s=W^{-}=0$, which is a contradiction. 
\end{proof}

\vspace{50pt}
$\mathbf{Acknowledgements}$: The author would like to thank Prof. Claude LeBrun for helpful comments and his articles [18], [19], [20].

\end{document}